\documentclass[twoside, 12pt]{article}
\usepackage{amsmath, amsthm, amssymb, amscd, mathptmx}
\usepackage{amsfonts}

\usepackage{pstricks}

\usepackage[all]{xy}\CompileMatrices\SelectTips{cm}{12}

\theoremstyle{plain}
\newtheorem{Thm}{\sc Theorem}[section]
\newtheorem{Theorem}[Thm]{\sc Theorem}
\newtheorem{Corollary}[Thm]{\sc Corollary}
\newtheorem*{Corollary*}{\sc Corollary}
\newtheorem{Proposition}[Thm]{\sc Proposition}
\newtheorem*{Proposition*}{\sc Proposition}
\newtheorem{Lemma}[Thm]{\sc Lemma}
\newtheorem{Question}[Thm]{\sc Question}

\theoremstyle{definition}

\theoremstyle{remark}
\newtheorem{Remark}[Thm]{Remark}
\newtheorem{Remarks}[Thm]{Remarks}

\newtheorem*{Example*}{Example}
\newtheorem*{Remark*}{Remark}

\newcommand{\FF}{{\mathbb F}}

\renewcommand{\O}{{\cal O}}

\newcommand{\NN}{{\mathbb N}}

\newcommand{\PP}{{\mathbb P}}

\newcommand{\End}{{\mathop{ End}\,}}

\newcommand{\Ext}{{\mathop{{\rm E}xt}}}

\newcommand{\coker}{\mathop{\rm coker}}

\mathchardef\mhyp="2D

\begin{document}

\title{
Nef line bundles over finite fields}
\author{Adrian Langer}
\date{November 27, 2011}

\maketitle


{{\sc Address:}\\
Institute of Mathematics, Warsaw University, ul.\ Banacha 2,
02-097 Warszawa, Poland\\}

\begin{abstract}
We use Totaro's examples of non-semiample nef line bundles
on smooth projective surfaces over finite fields to construct nef line bundles
for which the first cohomology group cannot be killed by any generically
finite covers. This is used to show a similar example of a nef and big
line bundle on a smooth projective threefold over a finite field.
This improves some examples of Bhatt and answers some of his questions.

Finally, we prove a new vanishing theorem for the first cohomology group of
strictly nef line bundles on projective varieties defined over finite fields.
\end{abstract}

\let\thefootnote\relax\footnote{The author is supported by the Bessel
Award of the Humboldt Foundation and a Polish MNiSW grant
(contract number N N201 420639). }

\section*{Introduction}

The main aim of this note is to see what kind of vanishing
theorems one can expect for line bundles on varieties defined over
fields of positive characteristic. In \cite{Bh} B. Bhatt proved
that for a semiample line bundle $L$ on a proper variety $X$ in
positive characteristic, one can kill cohomology $H^i(X, L)$ for
$i>0$ by passing to a finite cover of $X$. He also gave an example
of a nef line bundle $L$ on a curve defined over an uncountable
field and such that $H^1(L)$ cannot be killed by passing to finite
covers. We show that in fact this is true for a very general line
bundle, which is not semiample (see Theorem \ref{very-general}).

Since any nef line bundle on a curve defined over a finite field
is semiample, this example does not work over $\bar{\FF}_p$ and
Bhatt asked if the fact is true over such fields (see \cite[Remark
6.9]{Bh}). Here we use Totaro's examples of non-semiample line
bundles on surfaces defined over $\bar \FF_p$ (see \cite{To}) to
show that some multiple of such line bundle gives a nef line
bundle $L$ for which $H^1(L)$ cannot be killed by passing to
generically finite covers (see Theorem \ref{Totaro}). We also show
an example of a nef and big line bundle $L$ on a smooth projective
$3$-fold defined over $\bar \FF _p$, for which $H^1(L)$ cannot be
killed by passing to finite covers (see Proposition
\ref{Totaro2}). Since any nef and big line bundle on a normal
projective surface defined over $\bar \FF _p$ is semiample, this
is also the lowest possible dimension where one can find such an
example.

In analogy to the Kawamata--Viehweg vanishing theorem it is more
natural to kill, by passing to finite covers, cohomology $H^i(X,
L^{-1})$ for  nef and big line bundle $L$ and $i<\dim X$. This can always be
done in case of projective curves and surfaces (and in fact,
assuming normality, it is sufficient to use only Frobenius
morphisms). In case $L$ is semiample and big this was done in
\cite[Proposition 6.3]{Bh}, but the general nef and big case is
open. In Section \ref{van-s} we show that $H^1(X, L^{-1})$ can be
killed by some Frobenius pull back for a strictly nef line bundle
$L$ on a smooth projective variety defined over a finite field
(see Theorem \ref{surface-vanishing}). This is no longer true in
case of Totaro's example but it is consistent with the possibility
that every strictly nef line bundle on a smooth projective surface
defined over $\bar \FF _p$ is ample (this is Keel's question
\cite[Question 0.9]{Ke}).

There are a few theorems on line bundles on varieties defined over
an algebraic closure of a finite field that are no longer  valid
over other algebraically closed fields of positive characteristic.
For example, Artin proved that a nef and big line bundle on a
smooth projective surface over $\bar \FF _p$ is semiample. A basic
tool used in proofs of such facts is that numerically trivial line
bundles on a projective variety defined over $\bar \FF _p$ are
torsion. An interesting point in proof of our vanishing theorem is
that we use an analogous fact in the higher rank case, where it
follows from boundedness of the family of semistable vector
bundles with fixed Chern classes.

\medskip

The structure of the paper is as follows. In Section 1 we prove
that $H^1$ of a very general degree $0$ line bundle on a curve
cannot be killed by finite covers. In Section 2 we show an example
of a nef line bundle on a smooth projective surface over $\bar \FF
_p$ such that $H^1(L)$ cannot be killed by generically finite
covers. In Section 3 we use it to show a similar example of a nef
and big line bundle on a smooth $3$-fold defined over $\bar \FF
_p$. In Section 4 we recall vanishing theorems in positive
characteristic and ask some questions. Finally, in Section 5 we
show a new vanishing theorem for strictly nef line bundles on
varieties defined over $\bar \FF _p$.

\section{Line bundles on curves}

Let $X$ be an algebraic $k$-variety. We say that a \emph{very general point} of $X$ satisfies
some property if there exists a countable union of proper subvarieties
of $X$ such that the property is satisfied for all points outside of this
union.

The main aim of this section is proof of the following theorem
generalizing \cite[Example 6.8]{Bh}:

\begin{Theorem}\label{very-general}
Let $C$ be a smooth projective curve of genus $g\ge 2$ over an
algebraically closed field $k$ of characteristic $p>0$. Then for a
very general line bundle $M$ of degree $0$ on $C$ for any finite
surjective map $\pi : C'\to C$ the pullback map $H^1(C, M)\to
H^1(C', \pi ^* M)$ is injective.
\end{Theorem}

In particular, if the field $k$ is uncountable then on every curve $C/k$
there exist line bundles $M$ satisfying the assertion of the theorem.
If the field is countable there are no such line bundles.

Before proving this theorem we need some auxiliary facts explaining
conditions appearing in \cite[Example 6.8]{Bh}.

Let $E$ be a numerically flat bundle on a smooth projective
$k$-variety $X$. Let $\langle E \rangle$ denote the full tensor subcategory
of the category of numerically flat bundles on
$X$ spanned by $E$. Then we have the following lemma:

\begin{Lemma}\label{torsion}
Let $M$ be a numerically trivial line bundle on $X$. Then the
following conditions are equivalent:
\begin{enumerate}
\item $M$ is torsion.
\item $M$ is semiample.
\item There exists a non-negative integer $e$ and a finite \'etale
cover $h:X'\to X$ such that $M^{p^e}$ is isomorphic to a subsheaf
of $h_*\O_{X'}$.
\item There exists a non-negative integer $e$ and a finite \'etale
cover $h:X'\to X$ such that $M^{p^e}$ is an object of $\langle
h_*\O_{X'}\rangle $.
\end{enumerate}
\end{Lemma}

\begin{proof}
Equivalence of (1) and (2) is clear as the only globally generated
numerically trivial line bundle is trivial.

If $M$ is a torsion line bundle then for some non-negative integer
$e$ the line bundle $M^{p^e}$ is torsion of order prime to $p$.
Let $h:X'\to X$ be the cyclic covering associated to $M^{p^e}$.
Then we have $h^*(M^{p^e})\simeq \O_{X'}$ and hence
$M^{p^e}\hookrightarrow h_*h^*(M^{p^e})\simeq h_*\O_{X'}$, which
shows one implication $(1)\Rightarrow (3)$. Clearly, (3) implies
(4).

To show that (4) implies (1) let us take a finite \'etale cover
$h:X'\to X$. Since $h$ is \'etale, the diagram
$$ \xymatrix{ &X'\ar[d]^{h}\ar[r]^{F_{X'}}& X'\ar[d]^{h}\\
&X\ar[r]^{F_X}&X }$$ is cartesian (see, e.g., \cite{SGA5}). Since
$X$ is smooth, $F_X$ is flat. By flat base change we have
isomorphisms $F_X^*(h_*\O_{X'})\simeq h_*(F_{X'}^*\O_{X'})\simeq
h_*\O_{X'}$. Therefore by the Lange--Stuhler theorem (see
\cite{LSt}) the bundle $h_*\O_{X'}$ is \'etale trivializable,
i.e., there exists a finite \'etale morphism $h':X''\to X$ such
that $(h')^*(h_*\O_{X'})$ is trivial. Now if $M^{-p^e}$ is an
object of $\langle h_*\O_{X'}\rangle $ then $(h')^*M^{-p^e}$ is an
object of $\langle (h')^*h_*\O_{X'}\rangle =\langle
\O_{X''}\rangle$. But $\langle \O_{X''}\rangle$ contains only
trivial objects, so $(h')^*M^{-p^e}\simeq \O_{X''}$. But then
$M^{p^e}$ is a torsion line bundle of order prime to $p$ (this
follows, e.g., from \cite{BD}).
\end{proof}

The following corollary is a reformulated version of \cite[Example 6.8]{Bh}:

\begin{Corollary} \label{Bhatt}
Let $C$ be a smooth projective curve of genus $g\ge 2$ over an
algebraically closed field $k$ of characteristic $p>0$.
Let $M$ be a line bundle of degree $0$ on $C$.
Assume that $M$ is non-torsion and for every positive integer $e$ the map
$(F_C^e)^*: H^1(C, M)\to H^1(C, M^{p^e})$ induced by $e$-th Frobenius morphism
is injective. Then for any finite surjective map $\pi : C'\to C$
the pullback map $H^1(C, M)\to H^1(C', \pi ^* M)$ is injective.
\end{Corollary}

\begin{proof}
Lemma \ref{torsion} says that a line bundle is torsion if and only
if it satisfies the condition (2) of \cite[Example 6.8]{Bh}.
Therefore our assumptions imply that $M$ satisfies conditions
(1) and (2) of \cite[Example 6.8]{Bh}. Then the argument given in
this example shows the required assertion.
\end{proof}

\medskip
We can also prove this corollary in a different way using elements of $H^1(M)$
to construct vector bundles. A similar approach is used later to prove
Theorems \ref{Totaro}  and \ref{surface-vanishing}.
It is easy to see that the statement is equivalent to the following:

\begin{Lemma}
Let $C$ be a smooth projective curve of genus $g\ge 2$ over an
algebraically closed field $k$ of characteristic $p>0$.
Let $M$ be a line bundle of degree $0$ on $C$.
Assume that for some generically \'etale morphism $\pi : C'\to C$
the pullback map $H^1(C, M)\to H^1(C', \pi ^* M)$ is not injective.
Then $M$ is torsion.
\end{Lemma}

\begin{proof}
Let us note that the proof of  \cite[Proposition 2.2]{MS} shows the following
fact. Let $\pi : C'\to C$ be a generically \'etale morphism and let $E$ be a simple
strongly semistable vector bundle on $C$. If $\pi ^*E$ is not simple then
$\pi$ factors through a finite \'etale morphism $\tau: C''\to C$ such that
$\tau ^*E$ is not simple.

Now if the pullback map $\pi^*$ is not injective on $H^1(C, M)$ then
there exists a non-trivial extension $\xi \in \Ext^1(\O_C, M)=H^1(C, M)$ giving
a short exact sequence
$$0\to M\to E\to\O_C\to 0$$
such that its pull back by $\pi^*$ splits. Since $M$ and $\O_C$ are line bundles
of the same degree, $E$ is strongly semistable.
Tensoring the above sequence by $E^*$
we get
$$0\to E\to \End E\to E^*\to 0.$$
This sequence implies that $E$ is simple, whereas $h^0(\pi^*\End E)\ge 2$.
Therefore by the above there exists a finite \'etale morphism $\tau: C''\to C$ such that
$\tau ^*E$ is not simple. Then the pull-back of the above sequence shows that either $h^0(\tau ^*E)\ge 1$
or $h^0(\tau ^*E^*)\ge 2$. In both cases we see that either $\tau ^*M$ is trivial
or $\tau ^*\xi=0$. In the first case $M$ is \'etale trivializable, so it is torsion.
In the second case the kernel of $H^1(C, M)\to H^1(C'', \tau ^* M)$ is non-trivial.
Then $H^0(C, M\otimes (\tau_*\O_{C''}/\O_C))\ne 0$, so $M^{-1}$ is an object
of  $\langle \tau_*\O_{C''}\rangle $ and $M$ is torsion by Lemma \ref{torsion}.
\end{proof}

\medskip

\emph{ Proof of Theorem \ref{very-general}.}

We can use the absolute Frobenius morphism $F_C:C\to C$ to define
the following short exact sequence
$$0\to \O_C\to (F_C)_*\O_C\to B\to 0.$$
By \cite[Theorem 4.1.1]{Re} we know that for a general line bundle
$M$ we have $H^0(C, B\otimes M)=0$. In fact, this is true for
every line bundle which does not lie in the zero section $Z$ of
the theta divisor $\theta_B$ defined by $B$. But then using the
above short exact sequence we get the following exact sequence
$$H^0(B\otimes M)=0\to H^1(M)\to H^1((F_C)_*(F_C^*M))=H^1(C, (F_C)^* M).$$
Therefore for a general line bundle $M$ the map $(F_C)^*: H^1(C,
M)\to H^1(C, M^{p})$ is injective.

Let $J_C$ denote the Jacobian of $C$. Then the pull back by the
$e$-th Frobenius morphism $F_C^e$ defines the Verschiebung
morphism $V^e: J_C\to J_C$, which is a finite morphism. Now if a
line bundle $M$ does not lie in the divisor $(V^e)^{-1}(Z)$ of $Z$
then the map $(F_C)^*: H^1(C, M^{p^e})\to H^1(C, M^{p^{e+1}})$ is
injective. Therefore if $M$ lies outside of $\bigcup_{e\ge
0}(V^e)^{-1}(Z)$ then for every positive integer $e$ the map
$(F_C^e)^*: H^1(C, M)\to H^1(C, M^{p^e})$ is injective.

Now let us note that the subset of torsion line bundles forms in
$J$ a countable union of points. Therefore by Corollary \ref{Bhatt}
a very general line bundle $M$ satisfies the
required assertion.

\section{Totaro's example}\label{Tot}

Let $X$ be a normal projective variety of dimension $n\ge 2$ defined over
an algebraically closed field $k$ of characteristic $p>0$.
Let $D$ be a connected reduced effective Cartier divisor on $X$ and set $L=\O_X(D)$.
Suppose that the restriction of $L$ to $D$ has order $p$ in the Picard
group of $D$ and no multiple of $D$ moves in its linear
series  on $X$ (i.e., $h^0(X, \O_X(mD))=1$ for $m>0$).
Let us also assume that $H^1(X, \O_X)=0$.

All the above conditions are satisfied in Totaro's example in \cite[Theorem 6.1]{To},
in which $X$ is a smooth projective surface and $k=\bar \FF_p$.

\begin{Theorem}\label{Totaro}
Let us take $M=L^{-p-1}$ or $M=L^{p-1}$.
Then for every generically finite surjective morphism
$\pi : Y\to X$ the induced map $H^1(X, M)\to H^1(Y, \pi^*M)$
is non-zero.
\end{Theorem}

\begin{proof}

Since $D$ is connected and reduced, we have $H^0(X, \O_D)=k$.
Therefore the short exact sequence
$$0\to \O_X(-D)\to \O_X \to \O_D\to 0$$
shows that $H^1(X, \O_X(-D))$ embeds into $H^1(X,\O_X)=0$.
Similarly, the short exact sequences
$$0\to \O_X(-(a+1)D)\to \O_X(-aD) \to \O_D(-aD)\to 0$$
show that $H^1(\O_X(-(a+1)D))\hookrightarrow  H^1(\O_X(-aD))$ for $a<p$.
Hence $H^1(\O_X(-aD))=0$ for $a\le p$.
On the other hand,  by assumption we have  $\O_D(-pD)\simeq \O_D$, so the above sequence
for $a=p$ shows that $h^1(\O_X(-(p+1)D))=1$.

Let $E$ be the vector bundle defined by a non-trivial extension
in $\Ext ^1 (\O_X((p+1)D), \O_X)\simeq H^1(X, L^{-(p+1)})$.
Since $\Ext ^1 (\O_X(pD), \O_X)=H^1(\O_X(-(p+1)D))=0$ we get a commutative diagram
$$ \xymatrix{
&0\ar[r]& \O_X\ar[r]&E\ar[r]&\O_X((p+1)D)\ar[r]&0\\
&0\ar[r]& \O_X\ar[u]\ar[r]&\O_X\oplus \O_X(pD)\ar[u]\ar[r]&\O_X(pD)\ar[u]_{+D}\ar[r]&0
}$$
In particular, we have an inclusion $j: \O_X(pD)\hookrightarrow E$. Note that the cokernel of this
inclusion is torsion-free. Otherwise, we can find a non-zero effective divisor $D'$
such that $\O_X(pD+D')\hookrightarrow E$. Composing this map with projection $E\to  \O_X((p+1)D)$
we see that $D'\le D$. Therefore $D'=D$ (here we use that $h^0(X, L)=1$) and the extension
defined by $E$ in $\Ext ^1 (\O_X((p+1)D), \O_X)$ is trivial, a contradiction.

Now comparing determinants and using normality of $X$ we see that $\coker j$ is of the form
$I_Z\otimes \O_X(D)$ for some subscheme $Z\subset X$ of codimension $\ge 2$.
Since $E$ is locally free, the short exact sequence
$$0\to \O_X(pD)\to E\to I_Z\otimes \O_X(D)\to 0$$
shows that $Z$ is a locally complete intersection of codimension $2$.

Let us consider the following commutative diagram
$$ \xymatrix{
&& 0&&&\\
&& \O_D((p+1)D)\ar[r]\ar[u]&0&&\\
&0\ar[r]& \O_X((p+1)D)\ar[r]\ar[u]& \O_X((p+1)D)\ar[r]\ar[u]&0&\\
&0\ar[r]& \O_X(pD)\ar[u]\ar[r]&E\ar[u]\ar[r]&I_Z\otimes \O_X(D)\ar[u]\ar[r]&0\\
& & 0\ar[u]\ar[r]&\O_X \ar[u]\ar[r]& I_Z\otimes \O_X(D)\ar[u]_{=}&\\
&&&0\ar[u]\ar[r]&0\ar[u]&
}$$
By the snake lemma we have a short exact sequence
$$0\to \O_X \to  I_Z\otimes \O_X(D)\to \O_D((p+1)D)\to 0.$$
Let us consider a sufficiently general complete intersection
$S=H_1\cap ...\cap H_{n-2}$  of very ample divisors $H_1,...,H_{n-2}$.
Restricting the above sequence to $S$ and using $\O_D((p+1)D)\simeq \O_D(D)$ we get
$$0\to \O_S\to I_{Z\cap S}\otimes \O_S(D)\to \O_{D\cap S}(D)\to 0.$$
Therefore $\chi (S, I_{Z\cap S}\otimes \O_S(D))=\chi (S, \O_S(D))$.
On the other hand, $Z\cap S$ is a finite set of points, so  $\chi (S, I_{Z\cap S}\otimes \O_S(D))=\chi (S, \O_S(D))- \deg (Z\cap S)$,
which implies that  $Z\cap S=\emptyset$.
Therefore $Z=\emptyset$ and we have a short exact sequence
$$0\to \O_X(pD)\to E\to \O_X(D)\to 0$$
defining an extension class in $\Ext ^1 (\O_X(D), \O_X(pD))\simeq H^1(X, L^{p-1})$.
Note that this sequence is non-split. Otherwise, we have $E\simeq \O_X(pD)\oplus \O_X(D)$.
Since $h^0( \O_X(pD))=h^0(\O_X(D))=1$, the map $\O_X\to E$ has to vanish on the divisor $D$,
contradicting the fact that its cokernel is torsion-free.

Now let us take a generically finite surjective morphism
$\pi : Y\to X$. Note that $h^0(Y,\pi^*L^m)=1$
for all $m\ge 0$. Indeed, if $h^0(Y, \pi^*L^{m_0})\ge 2$
then by the bilinear map lemma we have
$$h^0(Y, \pi^*L^{(l+1)m_0})\ge h^0(Y, \pi^*L^{lm_0})+h^0(Y, \pi^*L^{m_0})-1\ge ...\ge l+2$$
for $l\ge 1$. This implies that $\kappa (Y, \pi^*L)\ge 1$ and
hence by \cite[Theorem 5.13]{Ue} (the proof of which works in all
characteristics) we have $\kappa (X, L)\ge 1$, a contradiction.

To prove that the induced map $H^1(X, M)\to H^1(Y, \pi^*M)$
is non-zero for $M=L^{-p-1}$ or $M=L^{p-1}$,
it is sufficient to prove
that the sequences
$$0\to \O_Y\simeq\pi^*\O_X\to \pi^*E\to \pi^*\O_X((p+1)D)\to 0$$
and
$$0\to \pi^*\O_X(pD)\to \pi^*E\to \pi^*\O_X(D)\to 0$$
are non-split. If the first sequence splits then
the cokernel of the map $\pi^*\O_X(pD)\to \pi^*E\simeq \O_Y \oplus \pi^*\O_X((p+1)D)$ contains
torsion, a contradiction.
If the second sequence splits then, since
$h^0(Y,\pi^*\O_X(pD))=h^0(Y, \pi^*\O_X(D))=1$, the map $\O_Y\to \pi^*E$ has to
vanish along the divisor $\pi^* D$. But this contradicts torsion-freeness of its cokernel.
\end{proof}

\begin{Remark}\label{rem}
Note that the beginning of proof of the above theorem shows that
$H^1(\O_X(-aD))$ is non-zero for every $a\ge p+1$.
\end{Remark}

\section{A $3$-fold example}

Let us recall that any nef and big divisor $L$ on a projective surface $X$
over the algebraic closure of a finite field is semiample
(see \cite[Introduction]{To}). In this case there exists a finite surjective
morphism $\pi: Y\to X$ such that the induced map $H^1(X, L)\to H^1(Y, \pi^*L)$ is zero
(see \cite[Proposition 6.2]{Bh}).

Here we show that this is no longer true for $3$-folds:

\begin{Proposition} \label{Totaro2}
There exists a nef and big line bundle $L_W$ in a smooth projective
$3$-fold $W$ over $\bar \FF_p$ such that for every finite
surjective morphism $\pi: Z\to W$ the induced map $H^1(W, L_W)\to
H^1(Z, \pi^*L_W)$ is non-zero.
\end{Proposition}

\begin{proof}
Let $X$ and $L$ be as in Section \ref{Tot}. Let $A$ be any ample line bundle on $X$.
Let $p:W=\PP _X(E)\to X$ be the projectivization of $E=L^{p-1}\oplus A$ over $X$ (this $W$ is similar
to that in the proof of \cite[Theorem 7.1]{To}) and set $L_W=\O_{\PP _X(E)}(1)$. Then $L_W$ is a nef and big line bundle
(nefness is clear as $E$ is nef; bigness follows from $c_1(L_W)^3= A^2+(p-1)LA>0$ which can be checked using
the Leray--Hirsch theorem $c_1(L_W)^2-c_1(L_W)p^*c_1(E)+p^*(c_2(E))=0$).

Note that by construction there exists a section $s: X\to W$ such that $s^*L_W\simeq L ^{p-1}$.
Let $\pi: Z\to W$ be a finite surjective morphism and let $Y$ be an irreducible component
of $Z\times_W X$ dominating $X$. Let $s':Y\to W$ and $\pi':Y\to X$ denote the corresponding maps.
We have the following commutative diagram
$$
\xymatrix{H^1(W, L_W)\ar[r]^{\pi^*}\ar[d]^{s^*}& H^1(Z, \pi^*L_W)\ar[d]^{(s')^*} \\
H^1(X, L^{p-1})\ar[r]^{(\pi')^*}& H^1(X', (\pi')^*L^{p-1})\\
}
$$
in which the map $(\pi')^*$ is non-zero by Theorem \ref{Totaro}.
On the other hand, the Leray spectral sequence shows that
the natural injection
$$H^1(X, E)=H^1(X, p_*L_W)\to H^1(W, L_W)$$
is an isomorphism and hence $s^*$ is surjective. Therefore the map $\pi^*$ is non-zero.
\end{proof}

\begin{Remark}
Note that in the above example we have
$$\begin{array}{rcl}
H^2(W, L_W^{-m-2})&=&H^1(X, R^1p_*\O_{\PP (E)}(-m-2))\simeq H^1(X, (S^{m}E)^*\otimes (\det E)^{-1})\\
&=& \bigoplus _{i=0}^mH^1(X, L^{-(p-1)(m-i+1)}\otimes A^{-(i+1)} ).
\end{array}$$
for $m\ge 0$.
So we can always kill it by a large power of the Frobenius morphism.
\end{Remark}

\section{Vanishing theorems in positive characteristic}

Let us recall the following theorem (see \cite[Theorem 10]{Fu};
see also \cite[Theorem 2.22 and Corollary 2.27]{La} for effective
versions of this theorem):

\begin{Theorem}
Let $L$ be a nef and big line bundle on a normal projective
$k$-variety $X$ of dimension $\ge 2$. Then $H^1(X, L^{-m})=0$ for $m\gg 0$.
\end{Theorem}

Taking composition of normalization and suitable Frobenius morphisms, the above theorem
can be used to deduce the following corollary:

\begin{Corollary}
Let $L$ be a nef and big line bundle on a normal projective
$k$-variety $X$. Then there exists a finite surjective morphism $\pi: Y\to X$
such that $H^1(Y, \pi^*L^{-1})=0$. In particular, the map  $\pi^*: H^1(X, L^{-1})\to H^1(Y, \pi^*L^{-1})$
induced by $\pi$ is zero.
\end{Corollary}

This corollary answers positively \cite[Question 6.13]{Bh} in the
surface case. In \cite[pp. 526--527]{Fu} Fujita uses Raynaud's
counterexample to Kodaira's vanishing theorem in positive
characteristic to construct a nef and big line bundle $L$ on a
smooth projective threefold $X$ (in positive characteristic) such
that $H^2(X, L^{-m})\ne 0$ for all $m\gg 0$. However, one can
easily see that in this example the map induced by the $m$-th
Frobenius pull back on $H^2(X, L^{-1})$ vanishes for all $m\gg 0$.
So if $L$ is a nef and big line bundle on a projective variety
then killing $H^i(X, L^{-1})$ for $i<\dim X$ by Frobenius
morphisms is the best kind of vanishing theorem one can possibly
expect.

\begin{Question}
Let $L$ be a nef and big line bundle on a smooth projective variety defined over an algebraically
closed field of positive characteristic. Fix an integer $0\le i< \dim X$.
Is the map $H^i(X, L^{-1})\to H^i(X, L^{-p^m})$ induced by the $m$-th Frobenius pull back zero
for $m\gg 0$?
\end{Question}

By the above we know that an answer is positive if $i\le 1$. By Serre's vanishing theorem
we also know that $ H^i(X, L^{-p^m})=0$ for ample $L$ and $m\gg 0$.
But in general an answer to the above question
is not known even if $L$ is semiample and big (note that \cite[Proposition 6.3]{Bh}
uses finite surjective morphisms to kill $H^i(X, L^{-1})$).

\section{Vanishing theorem on varieties over finite fields}\label{van-s}

Let us recall that a line bundle $L$ on a variety $X$ is called \emph{strictly nef}
if it has positive degree on every curve in $X$.

\begin{Theorem}\label{surface-vanishing}
Let $X$ be a smooth projective variety of dimension $d\ge 2$
defined over $\bar \FF_q$ where $q=p^e$ for some $e\ge 1$.
Let $L$ be a strictly nef line bundle on $X$. Then for large $n$ the map
$H^1(X, L^{-1})\to H^1(X,(F_X^n)^*L^{-1})$ induced by the $n$-th Frobenius morphism
$F_X^n$ is zero.
\end{Theorem}

The following lemma is well-known, but we include a proof for completeness:

\begin{Lemma} \label{semiampleness}
Let $X$ be a normal projective surface defined over an algebraically closed field.
Let $L$ be a nef line bundle on $X$ such that $L^2=0$ and $\kappa(X, L)=1$. Then $L$
is semiample.
\end{Lemma}

\begin{proof}
Clearly, we can assume that $X$ is smooth. By assumption there
exists $m>0$ such that $h^0(X, L^{ m})\ge 2$. Let us write the
linear system $|L^{ m}|$ as a sum of its moving part $M$ and its
fixed part $B$. Since $|M|$ has only a finite number of base
points, we have $MB\ge 0$. Since $0=(M+B)^2=m ML+mBL$, we have
$ML=BL=0$, which implies $M^2=MB=0$. Now the rational map given by
$|M|$ goes into a curve. Resolving this map we see that $M^2=0$
implies that $|M|$ has no base points. Since $MB=0$ this implies
that $B$ is equivalent to the pullback of a divisor from $C$, so
some multiple of $B$ is also base point free.
\end{proof}

Let $X$ be a smooth projective surface defined over $\bar \FF_q$.
Let $L$ be a divisor on $X$ such that $L^2=0$. If $H^1(X,
L^{-1})\ne 0$ then there exists a non-trivial extension
$$0\to \O_X\to E\to L\to 0.$$

\begin{Lemma} \label{semiample}
If $E$ is strongly slope $H$-semistable for some ample divisor $H$
then $L$ is semiample.
\end{Lemma}

\begin{proof}
Using standard covering tricks we can find a generically finite
morphism $\pi:Y\to X$ from a smooth projective surface $Y$ such
that $\pi^*L\simeq \O _Y(2N)$ for some Cartier divisor $N$. Then
the bundle $E'=\pi^*E(-N)$ has trivial determinant and
$c_2(E')=0$. Hence the family $\{(F_Y^n)^*E'\}_{n\in \NN}$ is bounded (as it is a
family of $H$-semistable bundles with the same Chern classes; see
\cite[Theorem 2.1]{La}). Since we work over $\bar \FF_q$,
$E'$ and $Y$ can be defined over the same finite field. Then the sheaves
$(F_Y^n)^*E'$ are also defined over the same field, so boundedness implies that
there exist some integers $n>m\ge 0$ such that $(F_Y^n)^*E'\simeq (F_Y^m)^*E'$.
Therefore by the Lange--Stuhler theorem \cite{LSt} the bundle $(F_Y^m)^*E'$ is
\'etale trivializable. Summing up, we can find  a generically
finite morphism $\pi':Y'\to X$ from a smooth projective surface
$Y$ such that $(\pi ')^*L\simeq \O _{Y'}(2N')$ for some Cartier
divisor $N'$ and the bundle $E''=(\pi')^*E(-N')$ is trivial. But
$E''$ fits into an exact sequence
$$0\to \O_{Y'}(-N')\to E''\to \O_{Y'}(N')\to 0,$$
so $\O_{Y'}(N')$ is globally generated. In particular, the bundle
$(\pi')^*L$ is nef and hence $L$ is nef. Now if $\kappa (X, L)=1$
then by Lemma \ref{semiampleness} the line bundle $L$ is semiample.
Therefore we can assume that $\kappa (X, L)\le 0$. Then by
\cite[Theorem 5.13]{Ue} we have $\kappa (Y', (\pi')^*L)\le 0$.
Since $(\pi')^*L$ is globally generated, it must therefore be
trivial. This implies that $L$ is numerically trivial and, since
$X$ is defined over $\bar \FF _p$, the line bundle $L$ must be
torsion.
\end{proof}

\medskip

\emph{Proof of Theorem \ref{surface-vanishing}.}

The proof is by induction on the dimension $d$ starting with the most difficult case
$d=2$. In this case we can assume that $L^2=0$ as otherwise $L$ is ample by
the Nakai--Moischezon criterion and the theorem is clear.

If $H^1(X, L^{-1})\ne 0$ then there exists a non-trivial extension
$$0\to \O_X\to E\to L\to 0.$$
Since both $\O_X$ and $L$ have degree zero  with respect to $L$,
the bundle $E$ is slope $L$-semistable. Let us fix an ample
divisor $H$ and consider semistability of $E$ with respect to
$H_t=(1-t)L+tH$ for $0<t\le 1$.

If $E$ is not $H_t$-semistable then there exists a saturated line
bundle $M\hookrightarrow E$ such that $MH_t>\frac{1}{2} LH_t>0$.
Then the induced map $M\to E\to L$ is non-zero giving an effective
divisor $D\in |L\otimes M^{-1}|$. We have
$$1\le DL<DH_t=LH_t-MH_t<\frac{1}{2}LH_t=\frac{1}{2}t LH.$$
This shows that if $0\le t\le \frac{2}{LH}$ then $E$ is
$H_t$-semistable.

Let us assume that $E$ is not slope $H$-stable. Then there exists
some $0<t<1$ such that $E$ is strictly $H_t$-semistable, i.e., it
is an extension of rank $1$ sheaves of equal $H_t$-slope. But such
a bundle is strongly $H_t$-semistable. By Lemma \ref{semiample}
this implies that $L$ is semiample contradicting our assumption
(since $L^2=0$ any section of $L^m$ gives a curve $C$ with
$LC=0$).

Thus we proved that $E$ is slope $H$-stable for any ample
polarization $H$. The $n$-th Frobenius pull back of an extension
of $L^{-1}$ by $\O_X$ is an extension of $L^{-p^n}$ by $\O_X$ and
thus it is either trivial or slope $H$-stable. Since by Lemma
\ref{semiample} $E$ is not strongly $H$-semistable, it follows
that there exists some $n>0$ such that the extension given by $E$
pull backs by $F_X^n$ to the trivial extension. This shows that
for some, possibly larger, $n>0$ the map $H^1(X, L^{-1})\to
H^1(X,L^{-p^n})$ induced by the $n$-th Frobenius pull back is
zero.

Now let us assume that $d\ge 3$ and the assertion holds for varieties of dimension $<d$.
Let us take some very ample divisor $H$ such that $H^1(X, L^{-m}(-H))=0$
for all $m\ge 1$. Such $H$ exists by Fujita's vanishing theorem \cite[Theorem 1]{Fu}.
Using Bertini's theorem we can assume that $H$ is given by a smooth subvariety of $X$.
Then the assertion follows from the induction assumption and the following commutative diagram
$$
\xymatrix{
0\ar[r]&H^1(X, L^{-1})\ar[r]\ar[d]^{(F_X^n)^*}& H^1(H, L^{-1}|_H)\ar[d]^{(F_H^n)^*} \\
0\ar[r]& H^1(X, L^{-p^n})\ar[r] & H^1(H, L^{-p^n}|_H) \\
}
$$

\medskip

\begin{Remarks}
\begin{enumerate}
\item
Note that the proof of Lemma \ref{semiample} works in any
dimension (after some small changes) and it shows the following.
Let $X$ be a smooth projective variety of dimension $d\ge 2$
defined over $\bar \FF_q$. Let $L$ be a line bundle on $X$ such
that $L^2H^{d-2}=0$ for some ample divisor $H$. If $H^1(X,
L^{-1})\ne 0$ then consider a non-trivial extension
$$0\to \O_X\to E\to L\to 0.$$
If $E$ is strongly slope $AH^{d-2}$-semistable for some ample $A$ then either $L$ is torsion or
$\kappa (X, L)\ge 1$.
In both cases we can find a curve $C$ such that $LC=0$ (in the second case we can take
intersection of a divisor given by some section of $L^{m}$ for some $m>0$ with a general complete
intersection of some multiples of $H$). This can be used to omit the use of Fujita's vanishing
theorem in proof of Theorem \ref{surface-vanishing}.

\item
Let us recall Keel's question \cite[Question 0.9]{Ke}. Let $X$
be a smooth projective surface defined over $\bar \FF_p$. Let $L$
be a strictly nef line bundle on $X$. Does it imply that $L$ is ample,
or equivalently that $L^2>0$? Theorem \ref{surface-vanishing} shows that
in this case we get an expected vanishing theorem.

On the other hand, take $r^2$ points ($r>3$) on $\PP^2$ over a
finite field, blow them up and take $L=p^*\O_{\PP^2} (r)\otimes
\O(-E)$, where $p$ is the blow up and $E$ is the exceptional
divisor. Clearly, we have $L^2=0$. If one chooses points on an
irreducible degree $r$ curve then $L$ is nef, but it is not
strictly nef. Is it strictly nef for some choice of points? If so,
it would give a negative answer to Keel's question. One can use
Nagata's results to show that $L$ is strictly nef if we blow up
very general points, but this does not say anything over countable
fields.

\item
It is not known if there exists a strictly nef line bundle on a
smooth projective variety defined over some finite field, which is
not ample. Note that such line bundles cannot be semiample.
However, the only known examples of non-semiample nef line bundles
on normal projective varieties over finite fields (see \cite{To})
are not strictly nef. In \cite[Section 5]{Ke} Keel gives
Koll\'ar's example of a strictly nef, non-semiample line bundle
on a non-normal surface defined over a finite field.

\item
Theorem \ref{surface-vanishing} should be compared to Totaro's
example. There we can find a nef line bundle $M$ with $M^2=0$ on a
smooth projective surface over $\bar \FF_p$ such that $H^1(X,
M^{-1})\ne 0$ and the map $H^1(X, M^{-1})\to H^1(X,M^{-p^n})$
induced by the $n$-th Frobenius pull back is always injective (see
Theorem \ref{Totaro}). By Theorem \ref{surface-vanishing} in any
such example one can find a curve $C$ with $MC=0$.
\end{enumerate}
\end{Remarks}

\bigskip

\emph{\large \bf Acknowledgements.}

The author would like to thank the University of Duisburg-Essen
for its hospitality during writing this paper.

\end{document}